\documentclass[12pt]{amsart}
\usepackage{amscd,amsmath,amssymb,amsfonts}
\usepackage[cmtip, all]{xy}
\usepackage{enumerate}
\theoremstyle{plain}
\newtheorem{thm}[equation]{Theorem}

\newtheorem{lem}[equation]{Lemma}
\newtheorem{cor}[equation]{Corollary}

\newtheorem{prop}[equation]{Proposition}

\newtheorem{hyp}[equation]{Hypothesis}
\newtheorem{obs}[equation]{Observation}

\theoremstyle{definition}

\newtheorem{rmk}[equation]{Remark}

\numberwithin{equation}{section}

\newcommand{\isoarrow}{{~\overset\sim\longrightarrow}}

\bigskip

\newcommand{\fp}{{\frak{p}}}

\newcommand{\ZZ}{{\mathbb Z}}

\newcommand{\fg}{{\frak{g}}}
\newcommand{\fk}{{\frak{k}}}
\newcommand{\fh}{{\frak{t}}}
\newcommand{\fu}{{\frak{u}}}
\newcommand{\fq}{{\frak{q}}}

\newcommand{\tilc}{{\tilde{c}}}

\newcommand{\ra}{{~\rightarrow~}}

\newcommand{\QQ}{{\mathbb Q}}

\newcommand{\Qbar}{{\overline{\mathbb Q}}}

\newcommand{\RR}{{\mathbb R}}
\newcommand{\ad}{{\bold A}}
\newcommand{\af}{{{\bold A}_f}}

\newcommand{\CC}{{\mathbb C}}

\begin{document}

\title [Testing rationality of coherent cohomology]
{Testing rationality of coherent cohomology of Shimura varieties}

\author{Michael Harris}
\address{Institut de Math\'ematiques de
Jussieu, U.M.R. 7586 du CNRS; UFR de Math\'ematiques \\ Universit\'e Paris-Diderot Paris 7}

\thanks{The research leading to these results has received funding from the European Research Council under the European Community's Seventh Framework Programme (FP7/2007-2013) / ERC Grant agreement n¡ 290766 (AAMOT)}

\dedicatory{in memory of Ilya Piatetski-Shapiro}
\smallskip

\date{\today}
\thanks{ }

\begin{abstract} Let $G' \subset G$ be an inclusion of reductive groups whose real points have a
non-trivial discrete series.  Combining ergodic methods of Burger-Sarnak and the author with a positivity
argument due to Li and the classification of minimal $K$-types of discrete series,
due to Salamanca-Riba, we show that, if $\pi$ is a cuspidal automorphic representation
of $G$ whose archimedean component is a sufficiently general discrete series, then there is a cuspidal
automorphic representation of $G'$, of (explicitly determined) discrete series type at infinity, that pairs non-trivially
with $\pi$.  When $G$ and $G'$ are inner forms of $U(n)$ and $U(n-1)$, respectively, this result is used
to define rationality criteria for sufficiently general coherent cohomological forms on $G$.
\end{abstract}



\maketitle

\section*{Introduction}

For years I would meet Ilya Piatetski-Shapiro at conferences and mathematical institutes, but we only spoke
a few times, and our conversations were generally brief and to the point.  Shortly after I arrived at the
IAS in 1983 for the special year on automorphic forms -- I later came to suspect that Ilya was in part
responsible for my having been invited -- he asked me to come to his office to tell him what I had been
doing.  I explained that I had been working on defining arithmetic models of automorphic vector bundles,
to provide a general geometric framework for Shimura's characterization of arithmetic
holomorphic automorphic forms -- forms rational over number fields.  Ilya asked whether 
I had a similar characterization of arithmetic non-holomorphic 
coherent cohomology classes.  I replied that I had not thought about the question; he suggested that I
work on the problem; and that was essentially the end of the discussion.\footnote{He did ask one more question:  
was I married?  I was not, and I was surprised to discover that
he had advice for me on that matter as well.}

Non-holomorphic coherent cohomology of Shimura varieties is usually defined by 
automorphic representations whose archimedean components are in the non-holomorphic
discrete series.  Since Takeyuki Oda had also advised me to look at such automorphic
representations, I spent the next few years developing a theory that provides a geometric
framework for such cohomology classes, and in many cases I was able to find a characterization
of the classes that are rational over number fields.  But
this characterization, developed in \cite{H90}, is based on studying the cup products of
the classes in question with holomorphic arithmetic forms on Shimura subvarieties.  These cup products
are compatible with the rational structure of canonical models of Shimura varieties, and therefore
the rationality of these cup products is a necessary condition for arithmeticity of the class being tested.
The method of \cite{H90} shows that it is often also a sufficient condition, but
provides no information in the typical case, where all such cup products vanish for purely local reasons. 

 The present paper gives an optimal answer to Piatetski-Shapiro's question for coherent cohomology of
Shimura varietes attached to unitary groups, at least when the infinitesimal character is
sufficiently regular.  As in \cite{H90}, the rationality criterion is based on restriction to Shimura
subvarieties and integrating against classes on the latter, but this {\it period integral} 
in general is purely analytic and has nothing to do with algebraic geometry.
Nevertheless, the multiplicity one condition implied by the Gross-Prasad conjecture \cite{GGP} shows the
existence of canonical period invariants, called the {\it Gross-Prasad periods} in \cite{BB}, that have
the property that the proportionality of the period integral to the corresponding Gross-Prasad period
is a necessary and sufficient condition for arithmeticity.   When the infinitesimal character is
sufficiently regular, it can be shown that the period integrals do not all vanish.  Thus the criterion
provides a foolproof way to reduce arithmeticity on a Shimura variety attached to an inner form
of $U(n)$ to that on an inner form of $U(n-1)$, and by induction we obtain a general criterion.

The remarks of the preceding paragraph need to be qualified, however.  The Gross-Prasad periods
are defined in \cite{BB}, assuming a stronger version of the local Gross-Prasad conjecture at
archimedean places than is currently available.  It has been proved in \cite{SZ} that the
multiplicity one conjecture of Gross-Prasad is true for $U(n-1)$-invariant linear forms on the
smooth Frechet completions of moderate growth of discrete series representations of $U(n)\times U(n-1)$,
where here for $m = n, n-1$, $U(m)$ denotes any inner form of the compact unitary group $U(m)$.  In order
to define Gross-Prasad periods, one needs multiplicity one for invariant linear forms on the
corresponding Harish-Chandra modules; in other words, one needs to know the {\it automatic continuity}
of such invariant linear forms.  This remains an open problem, at least when $U(n-1)$ is
non-compact and the representation of $U(n)$ is not holomorphic, so in that sense the results of the present
paper are conditional.  

The methods of this paper are in large part the same as those of my article \cite{HL} with J. S. Li.
The main difference is that the earlier paper only considered rationality criteria based on
cohomological cup products.  In particular, we make heavy use of Li's idea to apply Flensted-Jensen's formula for the leading matrix coefficients
of discrete series representations to show weak inclusion of representations of $U(n-1)$ in restrictions
from $U(n)$.   We also use the method of
Burger-Sarnak, or the more elementary methods of \S 7 of \cite{H90}, 
to deduce non-vanishing of global pairings from local weak inclusion.

We note that the proof of the full Gross-Prasad conjecture for unitary groups
over archimedean fields would allow us to derive the main theorem without reference to the 
Flensted-Jensen formula.  Indeed, the full Gross-Prasad conjecture determines the discrete series representations
 of (an inner form of) $U(n-1)$ that are weakly contained in any given discrete series representation $\pi$ of
 (an inner form of) $U(n)$.  In most cases, this weak containment has nothing to do with the leading matrix
 coefficient of $\pi$.
The paper is organized in a progression from most general and conditional results -- much of the
argument is valid for general Shimura varieties -- to the most specific and unconditional.

\bigskip

I thank Dipendra Prasad for asking the question answered in this paper and Jeff Adams for directing me to the article
\cite{SR} of Salamanca-Riba.  I also thank the editors of this volume for inviting me to
contribute.

\medskip

\section{Restrictions of discrete series representations and global consequences}

Let $G$ be a connected reductive algebraic group over $\QQ$, $G' \subset G$ a reductive
subgroup.  Let $Z_G$ and $Z_{G'}$ denote their respective centers, $A_G$
and $A_{G'}$ the maximal $\RR$-split tori in $Z_G$ and $Z_{G'}$.  We assume in what follows
that $A_G$ is of finite index in $A_{G'}$.  By a discrete series representation of $G(\RR)$ we will mean an irreducible
Hilbert space representation $\pi$ of $G(\RR)$ such that, for some character $\xi$ of $G(\RR)$,
$\pi\otimes \xi$ is trivial on an open subgroup $A^0 \subset A_G(\RR)$, and such that $\pi\otimes \xi$ defines a square-integrable
representation of $G(\RR)/A^0$.   

Suppose $\pi$ and $\pi'$ are discrete series representations of $G(\RR)$ and $G'(\RR)$, respectively, and assume
that their central characters coincide on an open subgroup $A^0 \subset A_G(\RR)$. 
\begin{hyp}\label{weak}  The representation $\pi'$ is weakly contained in the restriction of $\pi$ to $G'(\RR)$; in other words,
there is a non-trivial Hilbert space pairing 
$$\pi \otimes \pi^{\prime,\vee} \ra \CC
$$
invariant under $G'(\RR)/A^0$.
\end{hyp}

Suppose $\pi$ occurs in the $L^2$-automorphic spectrum of $G$;
in other words, there is an equivariant homomorphism $\lambda:  \pi \ra C^{\infty}(\Gamma \backslash G(\RR))$, for
some congruence subgroup $\Gamma \subset G(\QQ)$, such that, for some character $\xi$
of $G(\RR)$ trivial on $\Gamma$ and for an open subgroup $A^0 \subset A_G(\RR)$ as above 
$\lambda(\pi)\otimes \xi \subset L^2(\Gamma\cdot A^0\backslash G(\RR))$.  
Theorem 1.1(a) of \cite{BS} then asserts that $\pi'$ is weakly contained in the $L^2$ automorphic spectrum of 
$G'$:  in other words

\begin{thm}\label{BS} (Burger-Sarnak)  Assume the pair $(\pi,\pi')$ satisfies hypothesis \ref{weak}.
For an appropriate character $\xi'$ of $G'(\RR)$,
$\pi'\otimes \xi'$ is the limit, in the Fell topology of unitary representations of $G'(\RR)/A^0$, of
$L^2$ automorphic representations of $\Gamma'\cdot A^0\backslash G'(\RR)$ as $\Gamma'$ varies over
congruence subgroups of $G'(\RR)$.
\end{thm}

If $\alpha$ is a Hecke character of $Z_G$, we let $L^2_{cusp,\alpha}(G(\QQ)\backslash G(\ad))$
denote the space of complex-valued functions $f$ on $G(\QQ)\backslash G(\ad)$ with central character $\alpha$ such that,
for an appropriate Hecke character $\xi$ of $G(\ad)$, $f\cdot \xi$ is trivial on an open subgroup $A_0 \subset A_G(\RR)$ and
defines a square integrable function on $G(\QQ)\cdot A_0 \backslash G(\ad)$.
By a {\it cuspidal automorphic representation} of $G$ we will mean alternatively an irreducible
Hilbert space component of 
$L^2_{cusp,\alpha}(G(\QQ)\backslash G(\ad))$ for some $\alpha$ or the corresponding
$(Lie(G)_{\CC},K_{\infty}) \times G(\af))$-module, for some maximal compact subgroup $K_{\infty} \subset G(\RR)$.
The first part of the following corollary follows by standard arguments.  The second part
is then obvious, and the third part is a consequence of the well-known fact that integrable
discrete series representations are isolated in the Fell topology (or of the Poincar\'e series argument in
\cite{H90}, \S 7).   Note that $\pi$ and $\pi'$ now
denote global automorphic representations.

\begin{cor}\label{BSadelic}  Let $\pi$ be a cuspidal automorphic representation of $G$ whose archimedean
component $\pi_{\infty}$ is a discrete series representation of $G(\RR)$.   Let $\pi'_{\infty}$ by a discrete series
representation of $G'(\RR)$, and assume the pair $(\pi_{\infty},\pi'_{\infty})$ satisfies hypothesis \ref{weak}.
\begin{itemize}
\item[(a)]  For an appropriate Hecke character $\xi'$ of $G'(\ad)$,
$\pi'_{\infty}\otimes \xi'_{\infty}$ is the limit, in the Fell topology of unitary representations of $G'(\RR)/A^0$, of
archimedean components of $L^2$ automorphic representations of $G'(\ad)/A^0$.

\item[(b)]  In particular, if $\pi'_{\infty}$ is isolated in the automorphic spectrum of $G'$, then there is an $L^2$-automorphic
representation $\sigma$ of $G'$ with archimedean component $\pi^{\prime,\vee}_{\infty}$ and automorphic
forms $f \in \pi$ and $f' \in \pi'$ such that
$$\int_{G'(\QQ)\cdot A^0\backslash G'(\ad)} f(g')f'(g') dg' \neq 0.$$

\item[(c)]  If $\pi'_{\infty}$ belongs to the integrable discrete series, then there exists $\sigma$ as in (b).
\end{itemize}

\end{cor}

\section{Restrictions of minimal types}

In what follows, $G$ is a connected reductive group over $\QQ$ such that $G(\RR)$ has a discrete series.
Our goal in this section is to state a simple algebraic condition on $\pi$ as in the statement of Hypothesis
\ref{weak} that guarantees that the condition of the hypothesis is verified.  We are not striving for
maximum generality, and the condition is certainly unnecessarily strong.

Choose maximal compact connected subgroups $K^*_{\infty} \subset G(\RR)$ and $K^{\prime,*}_{\infty} \subset G'(\RR)$, with
$K^{\prime,*}_{\infty} \subset K^*_{\infty}$.  Define $K_{\infty} = K^*_{\infty}\cdot Z_G(\RR)$, $K'_{\infty} = K^{\prime,*}_{\infty}\cdot Z_{G'}(\RR)$.
Let $\pi$ be a discrete series representation of $G(\RR)$.  It has a minimal, or Blattner $K_{\infty}$-type $\tau$, an irreducible
representation $\tau$ of $K_{\infty}$ with the property that $\tau$ occurs in $\pi$ with multiplicity $1$, and is minimal
in the sense that its highest weight is shortest among all $K_{\infty}$-types occurring in $\pi$.  

In the next section
we recall the formula for the highest weight of $\tau$ in terms of the Harish-Chandra parameter of $\pi$.
On the other hand, given an irreducible representation $\tau$ of $K_{\infty}$, the article \cite{SR} provides a 
necessary and sufficient criterion for $\tau$ to be the minimal $K_{\infty}$-type of a discrete series
representation (and more generally, of a cohomologically induced representation of the form $A_{\mathfrak{q}}(\lambda)$).
Choose a maximal torus $H \subset K_{\infty}$.  
Since $G(\RR)$ has discrete series, $H$ is also a maximal torus of $G$.  We use lower case gothic letters to denote (complexified
real) Lie algebras; thus $\mathfrak{k} = Lie(K_{\infty})_{\CC}$, $\mathfrak{h} = Lie(H(\CC))$, and so on.  Denote by $\Delta(\fk,\fh)$ (resp. $\Delta(\fg,\fh)$)the set of
roots of $\fh$ in $\fk$ (resp. $\fg$) and let $(\bullet,\bullet)$ be the pairing on $\fh^*$
induced by any non-degenerate $G(\RR)$-invariant pairing on $\fg$.  

Let $\mu$ be the highest weight of an irreducible representation $\tau$ of $K_{\infty}$, relative to a system
of positive roots of $\fh$.  Define
\begin{equation}  \mu' = \mu + \sum_{\alpha \in \Delta(\fk,\fh), (\alpha,\mu) > 0} \alpha. \end{equation}
The $\theta$-stable parabolic defined by $\mu'$ is the parabolic subalgebra $\fq(\mu')$ spanned by
$\fh$ and the root spaces $\fg_{\alpha}$ for $(\alpha,\mu') \geq 0$.  The following observation is clear

\begin{obs}\label{Borel} For $\mu$ sufficiently regular, $\fq(\mu')$ is a Borel subalgebra of $\fg$.
\end{obs}

Denote by $\fu(\mu')$ the unipotent radical of $\fq(\mu')$, and let  $2\rho(\fu(\mu')$ denote the sum of
the roots of $\fh$ occurring in $\fu(\mu')$.
The following proposition is a special case of a result of Salamanca-Riba.
\begin{prop}\label{SR} ((\cite{SR}, Proposition 4.1))  Define
$$\Delta(\fu(\mu'),\fh) = \{\alpha \in \Delta(\fg,\fh) ~|~ (\alpha,\mu') > 0 \}.$$
Then $\tau$ is the lowest $K_{\infty}$-type of a discrete series representation $\pi$ of $G(\RR)$ if and only if 
\begin{itemize}
\item[(a)] $\fq(\mu')$ is a
Borel subalgebra of $\fg$;
\item[(b)] 
$$(\mu',\alpha) \geq (2\rho(\fu(\mu')),\alpha)$$
for all $\alpha \in \Delta(\fu(\mu'),\fh)$ 
\end{itemize}
In that case, the Harish-Chandra parameter of $\pi$
is equal to $\mu' - 2\rho(\fu(\mu')).$
\end{prop}

When $\fq(\mu')$ is not a Borel subalgebra, Salamanca-Riba determines the parameters $\lambda$ for which $\tau$ is the lowest $K_{\infty}$-type
of a module $A_{\fq(\mu')}(\lambda)$, under a condition analogous to (b).  We let $G' \subset G$ be a reductive subgroup
as in the previous section, and we assume $G'(\RR)$ has a non-trivial discrete series.  In the next few propositions, a character $\lambda$ of $\fh$ is ``sufficiently regular" if
$|(\lambda,\alpha)| > N(G)$ for all roots $\alpha$, where $N(G)$ is a positive constant, depending only on $G(\RR)$, that can be made explicit.

\begin{cor}\label{mintype}  Let $\tau$ be an irreducible representation of $K_{\infty}$ with highest weight $\mu$.  For $\mu$ sufficiently regular, 
$\tau$ is the lowest $K_{\infty}$-type of a discrete series representation $\pi$ of $G(\RR)$.  Moreover, for $\mu$ sufficiently
regular, the restriction of $\tau$ to $K'_{\infty}$ contains an irreducible representation $\tau'$ that is the lowest $K'_{\infty}$-type
of a discrete series representation $\pi'$ of $G'(\RR)$.  
\end{cor}
\begin{proof}  The set of subsets of  $\Delta(\fg,\fh)$ is finite, so the term $2\rho(\fu(\mu'))$ that appears in (b) of \cite{SR} belongs to a
bounded subset of the vector space $Hom(\fh,\RR)$.  This implies the first part of the Corollary, and the second part follows by applying the first part to $G'$, bearing in mind
that the highest weight of $\tau$ is an extreme weight of a constituent of its restriction to $K'_{\infty}$.
\end{proof}

\begin{prop}\label{FlJ}  Let $\pi$ be an irreducible discrete series representation of $G(\RR)$ with Harish-Chandra parameter $\lambda_{\pi}$.
Suppose $\lambda_{\pi}$ is sufficiently regular.  Then there is a discrete series representation $\pi'$ of $G'(\RR)$ weakly
contained in the restriction of $\pi$ to $G'(\RR)$.  Moreover (possibly under a more stringent regularity condition) we can assume
$\pi'$ to be in the integrable discrete series.
\end{prop}

\begin{proof}  The proof is based on the argument of Li used in \cite{L90, HL}.  Let $\pi'$ be the discrete series representation
of $G'(\RR)$ mentioned in the statement of \ref{mintype}, and let $\tau'$ be its minimal $K'_{\infty}$-type.  Let
$\psi_{\pi}$ be the matrix coefficient of $\pi$ 
occurring in its
minimal $K$-type $\tau$, defined by formula (1.2.6) of \cite{HL}.  Proposition 1.2.3 of \cite{HL}, based on the calculations in
\S 4 of \cite{L90}, asserts that $\pi'$ is weakly contained in the restriction of $\pi$ provided three conditions are satisfied.
First, the restriction of $\psi_{\pi}$ to $G'$ has to be in $L^{2 + \epsilon}(H)$; the argument, valid for any tempered
representation of $G$, is reviewed in the course of the proof of Proposition 2.3.4 of \cite{HL}.  Next, $\psi_{\pi}$ has to satisfy the
Flensted-Jensen formula of \cite{FJ}; but this is true for all discrete series representations.  The final condition is that $\tau'$
be contained in the restriction to $K'_{\infty}$ of $\tau$; this is our initial hypothesis.

The final claim follows from the characterization of integrable discrete series by Hecht and Schmid \cite{HS}:  any discrete
series with sufficiently regular Harish-Chandra parameter is integrable.

\end{proof}

We return to the global situation:  $G \supset G'$ are connected reductive groups
over $\QQ$, and both $G(\RR)$ and $G'(\RR)$ are assumed to have discrete series.
Combining \ref{FlJ} with \ref{BSadelic}, we draw the following global conclusion.  

\begin{prop}\label{pairing}  Let $\pi$ be a cuspidal automorphic representation of $G(\ad)$, with archimedean component $\pi_{\infty}$.  Suppose the Harish-Chandra
parameter of $\pi_{\infty}$ is sufficiently regular.  Then there is an integrable discrete series representation $\pi'_{\infty}$ of $G'(\RR)$, a cuspidal automorphic
representation $\pi'$ of $G'$ with archimedean component $\pi^{\prime}_{\infty}$ and automorphic
forms $f \in \pi$ and $f' \in \pi'$ such that
$$\int_{G'(\QQ)\cdot A^0\backslash G'(\ad)} f(g')f'(g') dg' \neq 0.$$
Moreover, $f$ and $f'$ can be taken in the minimal $K$- (resp. $K'$-) types of the archimedean components of their respective representations.
\end{prop}

\begin{proof}  We are changing notation:  the representation $\pi'$ here is $\sigma^{\vee}$ in
\ref{BSadelic}, and we have $\pi'_{\infty}$ pairing with rather than weakly contained in $\pi_{\infty}$.
Since \ref{BSadelic} only claims that $\sigma$ is an $L^2$-automorphic representation, we need to show that in fact $\pi'$ is
cuspidal.  If not, then $\pi'$ is a residual square-integrable representation.  But since $\pi'_{\infty}$
is tempered, this contradicts Wallach's theorem \cite{Wa} that residual representations are non-tempered at all places.
\end{proof}

If we assume that automorphic representations with discrete series components at infinity are isolated in the automorphic spectrum, then
we can relax the regularity condition in the last proposition, using the Burger-Sarnak theorem.

\section{Coherent cohomology of unitary group Shimura varieties}

\subsection*{Preliminary notation}

Let $F^+$ be a totally  real field, $F$ a totally imaginary quadratic
extension of $F^+$.  
The quadratic Hecke character of $F^+$ attached to the extension $F$ is denoted
$\varepsilon_F$.
 Let $W$ be an $n$-dimensional $F$-vector space, endowed with a non-degenerate
hermitian form $<\bullet,\bullet>_W$, relative to the extension $F/F^+$.
We let $\Sigma^+$, resp. $\Sigma_F$, denote the set of complex embeddings
of $F^+$, resp. $F$, and choose a CM type $\Sigma \subset \Sigma_F$, i.e.
a subset which upon restriction to $F^+$ is identified with $\Sigma^+$.
Complex conjugation in $Gal(F/F^+)$ is denoted $c$.

The hermitian pairing $\langle \bullet,\bullet\rangle _W$ defines an involution $\tilc$
on the algebra $End(W)$ via
\begin{equation*}
\langle a(v),v'\rangle _W = \langle v,a^{\tilc}(v')\rangle _W, \tag{0.1}\end{equation*}%
and this involution extends to $End(W\otimes_{\QQ}R)$ for any $\QQ$-algebra $R$.
We define the
algebraic group $U(W) = U(W,\langle \bullet,\bullet\rangle_W)$ 
(resp $GU(W) = GU(W,\langle \bullet,\bullet\rangle_W)$)
over $F^+$ (resp. $\QQ$) such that, for any $F^+$-algebra $R$ (resp. $\QQ$-algebra $S$)
\begin{equation*}U(W)(R) = \{g \in GL(W\otimes_{F^+} R)~|~ g\cdot\tilc(g) = 1 \};  \end{equation*}
\begin{equation*}GU(W)(S) = \{g \in GL(W\otimes R)~|~ g\cdot\tilc(g) = \nu(g) \text{ for some } \nu(g) \in R^{\times}\};  \end{equation*}

The quadratic Hecke character
of $\ad_{F^+}^{\times}$ corresponding to the extension $F/F^+$ is denoted
\begin{equation*}\eta_{F/F^+}:  \ad_{F^+}^{\times}/F^{+,\times}N_{F/F^+} \ad_{F}^{\times} \isoarrow \pm 1. \end{equation*}%

If $W$ is a $1$-dimensional hermitian or skew-hermitian
space then we write $U(W) = U(1)$.  The $F^+$-rational points of $U(1)$ are
given by the group of elements $x \in F^{\times}$ with $N_{F/F^+}(x) = 1$.

Representations of groups over archimedean local fields are assumed to be
admissible smooth Fr\'echet representations of moderate growth, in the sense of Casselman and Wallach (cf. \cite{C}),
and linear forms on such representations are assumed to be continuous with respect to the Fr\'echet topology.
The relation between such linear forms and linear forms on the corresponding Harish-Chandra modules
is discussed in the appropriate place.  The Harish-Chandra $(Lie(G),K_{\infty})$-module attached to a smooth Fr\'echet representation 
$\Pi$ of the group $G$ (and a choice of $K_{\infty}$) is denoted $\Pi^{(0)}$.  Similarly, if $\Pi$ is a representation
of an adelic group $G(\ad)$ of the form $\Pi_{\infty}\otimes \Pi_f$, where $\Pi_f$ is an admissible representation
of $G(\af)$ and $\Pi_{\infty}$ is an admissible smooth Fr\'echet representation (of moderate growth) of the archimedean
part of $G(\ad)$, then we write $\Pi^{(0)} = \Pi_{\infty}^{(0)}\otimes \Pi_f$.

\subsection{Multiplicity one and Gross-Prasad periods}

Let $W$ be an $n$-dimensional hermitian space as above, and write $W$ as the orthogonal direct
sum $W = W' \oplus W_0$ over $F$, where $\dim W' = n-1$.  Let $G = GU(W)$ as above, and let
$G' = GU(W') \times GU(W_0) \cap G$; it is the subgroup of $(g',g_0) \in GU(W') \times GU(W_0)$ with
$\nu(g') = \nu(g_0)$.  In this section we apply the results of \S 2 to the inclusion $G' \subset G$.
The new feature in this special case is that when $\pi \otimes \pi'$ is a cuspidal automorphic representation
of $G \times G'$, the space of global pairings $\pi \otimes \pi' \ra \CC$ is of dimension at most $1$.  

Let $X$ and $X'$ be $G(\RR)$ and $G'(\RR)$-homogeneous hermitian symmetric domains, respectively,
so that $(G',X') \hookrightarrow (G,X)$ is a morphism of Shimura data.
We choose maximal connected compact (mod center) subgroups $K \subset G(\RR)$ and $K' \subset G'(\RR)\cap K$
as in \S 2.  Then $K = Stab_{G(\RR)}(x)$ for some $x \in X' \subset X$, and we assume $x$ to be a CM point (the fixed point in $X'$ of
a $\QQ$-rational torus).  We are thus able to apply the theory of \cite{BB}.  Suppose $\pi$ and $\pi'$ are cuspidal automorphic
representations of $G$ and $G'$ respectively.  As in \S 4.3 of \cite{BB}, there are number fields $E(\pi)$ and $E(\pi')$ (depending on $x$) over which
$\pi^{(0)}$ and $\pi^{\prime,(0)}$ have rational structures, denoted $\pi^{(0)}_{E(\pi)}$ and  $\pi^{\prime,(0)}_{E(\pi')}$.  

As in \cite{BB}, we let $H \subset G$ and $H' \subset G'$ denote the unitary subgroups, $\mathfrak{h}' = Lie(H')$.  We define
$$L_0(\pi,\pi') = Hom_{(\mathfrak{h}',K')\times H'(\af)}(\pi^{(0)}\otimes \pi^{\prime,(0)},\CC).$$
We assume the {\it automatic continuity hypothesis} (Conjecture 4.3.1 of \cite{BB}):  every element of $L_0(\pi,\pi')$ has a natural
continuous extension to the admissible smooth Fr\'echet representations of moderate growth associated to $\pi$ and $\pi'$.
It then follows from \cite{SZ} that $L_0(\pi,\pi')$ is
a space of dimension at most $1$ (cf. \cite{BB}, Corollary 4.3.2).  We also consider the variant
$$L_0^{mot}(\pi,\pi') = L^{mot}_0(\pi,\pi') = Hom_{(\frak{h}',K')\times H'(\af)}(\pi^{mot}\otimes \pi^{\prime,mot},E(\pi)\otimes E(\pi')\otimes \CC)
$$
where
$\pi^{mot} = \pi^{(0)}_{E(\pi)}\otimes_{\QQ}\CC$, $\pi^{\prime,mot} = \pi^{\prime,(0)}_{E(\pi'}\otimes_{\QQ}\CC.$
Corollary 4.3.8 of \cite{BB} asserts that $L_0^{mot}(\pi,\pi')$ is either $0$ or a free $E(\pi)\otimes E(\pi')\otimes \CC$-module
of rank $1$ (there is a misprint in {\it loc. cit.}), and it follows as in the beginning of \S 4.4 that
\begin{prop}  Assume $L_0^{mot}(\pi,\pi') \neq 0$.  Then the free rank $1$ $E(\pi)\otimes E(\pi')\otimes \CC$-module $L_0^{mot}(\pi,\pi')$ has a generator $I^{mot}(\pi,\pi')$
such that 
$$I^{mot}{(\pi,\pi')}( \pi^{(0)}_{E(\pi)}\otimes \pi^{\prime,(0)}_{E(\pi')}) = E(\pi)\otimes E(\pi').$$
\end{prop}
(This corrects another misprint in {\it loc. cit.}.)

On the other hand, the integral (cf. \cite{BB}, (4.2.3))
\begin{equation} f \otimes f'  \mapsto  \int_{H'(F)\backslash H'(\ad)} f(h')f'(h') dh' \end{equation}
defines elements of 
$$L_0(\pi^{(0)}_{E(\pi)}\otimes_{E(\pi),\alpha}\CC,\pi^{\prime,(0)}_{E(\pi')}\otimes_{E(\pi'),\alpha}\CC)$$
as $\alpha$ varies over $Hom(E(\pi)\otimes E(\pi'),\CC)$, and thus 
an element $I^{can}(\pi,\pi') \in L_0^{mot}(\pi,\pi')$ (mislabelled $I^{mot}$ in (4.4.1) of \cite{BB}).   The
{\it Gross-Prasad period invariant} is the constant
$$P(\pi,\pi') \in (E(\pi)\otimes E(\pi')\otimes \CC),$$ 
well-defined up to a factor
in $(E(\pi)\otimes E(\pi'))^{\times}$,
such that
$$I^{can}(\pi,\pi') = P(\pi,\pi')\cdot I^{mot}(\pi,\pi').$$
Again, the definition of the Gross-Prasad period is conditional on the automatic continuity hypothesis.

In {\it loc. cit.} a conjecture is assumed that implies that either $P(\pi,\pi') = 0$ or $P(\pi,\pi') \in (E(\pi)\otimes E(\pi')\otimes \CC)^{\times}$.  
In order to avoid referring to this conjecture -- and to avoid creating new misprints -- we henceforward fix $\alpha \in Hom(E(\pi)\otimes E(\pi'),\CC)$
of $E(\pi)\cdot E(\pi')$ and consider $\pi^{(0)}_{E(\pi)}$ and  $\pi^{\prime,(0)}_{E(\pi')}$ as subspaces of the respective
complex automorphic representations.    Then we let $P_{\alpha}(\pi,\pi') \in \CC$ be the projection of $P(\pi,\pi')$ on the
$\alpha$-component; $P_{\alpha}(\pi,\pi') = 0$ if and only if $ \int_{H'(F)\backslash H'(\ad)} f(h')f'(h') dh' = 0$ for all
$f \in \pi$, $f' \in \pi'$.

We reformulate \ref{pairing} in terms of the Gross-Prasad periods.

\begin{prop}\label{GPpairing}  Let $\pi$ be a cuspidal automorphic representation of $G(\ad)$, with archimedean component $\pi_{\infty}$ in the discrete series.  Suppose the Harish-Chandra
parameter of $\pi_{\infty}$ is sufficiently regular.  Then there is an integrable discrete series representation $\pi'_{\infty}$ of $G'(\RR)$, a cuspidal automorphic
representation $\pi'$ of $G'$ with archimedean component $\pi'_{\infty}$ such that $I^{can}(\pi,\pi') \neq 0.$
Moreover, we can find $f \in \pi$ and $f' \in \pi'$  in the minimal $K$- (resp. $K'$-) types of the archimedean components of their respective representations
such that $I^{can}(\pi,\pi')(f,f') \neq 0.$
\end{prop}

\begin{cor}\label{ratpairing}  Let $\pi$ and $\pi'$ be as in the previous proposition and $\alpha$ as in the discussion
above.  Then $P_{\alpha}(\pi,\pi')^{-1}\cdot I^{can}(\pi,\pi')$ is a non-zero $\alpha(E(\pi)\otimes E(\pi'))$-bilinear pairing
$$ \pi^{(0)}_{E(\pi)}\otimes \pi^{\prime,(0)}_{E(\pi')}) \rightarrow  \alpha(E(\pi)\otimes E(\pi'))$$
that does not vanish on the tensor product of the minimal types at infinity.
\end{cor}

Let $\pi^{(0)}_{\Qbar}$ and  $\pi^{\prime,(0)}_{\Qbar}$ 
be the $\Qbar$-subspaces generated by $\pi^{(0)}_{E(\pi)}$ and  $\pi^{\prime,(0)}_{E(\pi')}$ in  the respective
complex automorphic representations.  

\begin{thm}\label{mainthm}  Assume the automatic continuity hypothesis (Conjecture 4.3.1 of \cite{BB}) Let $\pi$ be a cuspidal automorphic representation of $G(\ad)$, with archimedean component $\pi_{\infty}$ in the discrete series and
minimal $K$-type $\tau$.  Suppose the Harish-Chandra
parameter of $\pi_{\infty}$ is sufficiently regular.   Let $f \in \pi$ be an automorphic form whose archimedean component lies in $\tau$; in particular,
$f \in \pi^{(0)}$.  Let $\pi'_{\infty}$ be a discrete series representation of $G'(\RR)$ that satisfies the conditions of \ref{GPpairing}.  
Then $f \in \pi^{(0)}_{\Qbar}$ if and only if, for every $\gamma \in G(\af)$, every 
automorphic representation $\pi'$ of $G'$ of infinity type $\pi'_{\infty}$, and every $f' \in \pi'_{\Qbar}$, 
\begin{equation}\label{condition} \lambda_{\gamma,f'}(f) := P_{\alpha}(\pi,\pi')^{-1}\cdot I^{can}(\pi,\pi')(f^{\gamma},f') \in \Qbar \end{equation}
where $f^{\gamma}(h) = f(h\gamma)$.

\end{thm}
\begin{proof}  The proof is exactly analogous to that of Theorem 7.6 of \cite{H90}.  We temporarily adopt the notation
of that proof:  thus $\pi$ contributes to the coherent cohomology space $\bar{H}^q([\mathcal{V}])$ for some
irreducible automorphic vector bundle $[\mathcal{V}]$ on the Shimura variety attached to $G$.  The space
$\bar{H}^q$ of \cite{H90} is elsewhere called {\it interior cohomology}; in \cite{BB} it is denoted $H^q_{!}$.
The regularity
hypothesis implies that the $\pi_f$-isotypic component of $\bar{H}^q([\mathcal{V}])$ consists entirely of the
classes of cusp forms (cf. \cite{BB}, 3.3.2); in particular, hypothesis (d) of \cite{H90} 7.6.  Hypotheses
(a) and (c) hold for similar reasons:  the point is that we can take $\fh$ to be a Cartan subalgebra of
$Lie(G')$ as well as $Lie(G)$, and the walls of a Weyl chamber for $G'$ form a subset
of the walls of a Weyl chamber for $G$, so if an infinitesimal character in $\fh^*$ is far from the walls for $G$
it is {\it a fortiori} far from the walls for $G'$.  Condition (b) of \cite{H90}, 7.6, which identifies $I^{can}$ as a cup product
in coherent cohomology, does not hold; instead we have
the condition that $\dim Hom_{(\mathfrak{h}',K')}(\pi^{(0)}_{\infty}\otimes \pi^{\prime,(0)}_{\infty},\CC) = 1$, which
requires automatic continuity.

Now the rational structure $\pi^{(0)}_{\Qbar}$ is invariant under the action of $G(\af)$, so it is obvious that
if $f \in \pi^{(0)}_{\Qbar}$ then (\ref{condition}) is satisfied for all $\pi'$ and all $f' \in \pi'_{\Qbar}$.  So it suffices
to show that the linear forms $\lambda_{\gamma,f'}$ span the space of linear forms on the (countable-dimensional) subspace
$\tau\otimes \pi_f$ of $\pi$.  This is the content of Theorem 7.4 of \cite{H90}; the condition $0 \neq p(v)$ in the statement of that
theorem is implied by the assertion concerning minimal types in \ref{GPpairing}.  

\end{proof}

\begin{rmk}  Again, if we know a priori that $\pi'_{\infty}$ is isolated in the automorphic spectrum of $G'$, then we can
apply \cite{BS} rather than the arguments in \cite{H90}.
\end{rmk}

\begin{rmk}  By using the more precise Gross-Prasad periods defined in \cite{BB}, we can obtain a criterion for
$f$ to be rational over the field $E(\pi)$, and not just over $\Qbar$.  The condition is that 
$Gal(\Qbar/E(\pi))$ acts on the quantities $\lambda_{\gamma,f'}(f)$ by permuting every appearance of $f$, $f'$,
$\pi$, $\pi'$, and $\alpha$, including in the period invariant appearing on the right-hand side of (\ref{condition}).
\end{rmk}

Theorem \ref{mainthm} characterizes $\Qbar$-rational forms on $G$ of (sufficiently regular) coherent cohomology type
by their pairings with $\Qbar$-rational forms on $G'$ of coherent cohomology type, but of course one needs to be able
to characterize the latter.  However,
it was noted during the proof that the regularity hypothesis for $\pi_{\infty}$ implied a regularity hypothesis for $\pi'_{\infty}$.
Indeed, it follows from the classical branching laws for the restriction of a representation of (compact) $U(m)$ to
$U(m-1)$ that conditions (a) and (b) of \ref{SR} for the minimal $K$-type $\tau$ imply the corresponding conditions for at least
one irreducible constituent $\tau'$ of its restriction to $K'$.   Thus by induction on $n$ we can assume that the $\Qbar$-rational 
$f'$ used to define the $\lambda_{\gamma,f'}$ have already been identified.  This gives a complete answer to
the question of Piatetski-Shapiro mentioned in the introduction in the case of Shimura varieties attached to unitary
groups, at least under a regularity hypothesis.  I don't know whether or
not the regularity hypothesis can be relaxed.

\begin{rmk}  The rationality criterion of Theorem \ref{mainthm} makes specific use of the representation $\pi^{\prime}_{\infty}$
that can be proved to pair non-trivially with $\pi_{\infty}$, thanks to the positivity property of the Flensted-Jensen
minimal matrix coefficients.   If we know independently that (the dual of) $\pi^{\prime}_{\infty}$ is weakly contained 
in $\pi_{\infty}$, then the analogue of Theorem \ref{mainthm} remains valid.  The pairings whose existence is predicted by
\cite{GGP} do not restrict in most cases to non-trivial pairings of minimal types; nevertheless, the 
Gross-Prasad conjecture for real unitary groups implies that each non-trivial pairing gives a new rationality criterion for coherent
cohomology.
\end{rmk}

\section{Gross-Prasad periods as cup products} 

The novel feature of this article is the introduction of a rationality criterion, valid for automorphic
representations of sufficiently regular discrete series type, that, unlike the criterion of \cite{H90, HL},
is not derived from a cup product in coherent cohomology.  In this final section we characterize
cup-product pairings among all pairings given by integrals over $H'(\ad)$, and show that
certain $\pi_{\infty}$  never admit cup product pairings with coherent cohomology classes on
the Shimura variety attached to $H'$.  Thus the rationality criterion of the present paper is substantially
more general than that of \cite{H90}.  

The cup-product pairings are of independent interest, since they can be used to study period relations.
This will be explained in forthcoming work.

\subsection{Parameters}

In this section we let $H$ denote the real Lie group $U(r,s)$, the unitary group of the standard 
hermitian form $I_{r,s} = \begin{pmatrix} I_r & 0 \\ 0 & -I_s \end{pmatrix}$.  The maximal compact subgroup
$K = U(r) \times U(s)$, the subgroup that respects the decomposition of $\CC^r$ into positive-definite and negative-definite
subspaces as indicated by the form of $I_{r,s}$.
For a maximal torus $T$ we take the group of diagonal matrices in $H$; this is also a maximal torus of $K$.
An irreducible finite-dimensional representation $(\sigma,V)$ of
$H$ is determined by its highest weight $a_{\sigma}$ relative to $T$ and the upper-triangular Borel algebra $\mathfrak{b}$ of 
$\mathfrak{h} = Lie(H)_{\CC}$; $a_{\sigma}$ is written in the usual way as a non-increasing
$n$-tuple of integers $(a_1 \geq a_2 \geq \dots \geq a_n)$.   

Let $\rho = (\frac{n-1}{2},\frac{n-3}{2},\dots, \frac{1-n}{2})$ be the half-sum of positive roots for $\mathfrak{b}$.  We let
$\tilde{\rho} = \rho + \frac{n-1}{2}(1,1,\dots,1) = (n-1,n-2,\dots,1,0)$.   The infinitesimal character of the representation
$(\sigma,V)$ above is the element $\lambda_{\sigma} = a_{\sigma} + \rho$ of $\fh^*$; the element
$\tilde{\lambda}_{\sigma} = a_{\sigma} + \tilde{\rho} = (a_i + n-i, i = 1, \dots, n)$ will be called the {\it Hodge parameter}, for reasons that will soon be clear.

We consider the $L$-packet $\Pi_{\sigma} = \Pi^{r,s}_{\sigma}$ of discrete series representations $\pi$ of $H$ such that
$H^i(\fh,K;\pi\otimes V^{\vee}) \neq 0$ for some $i$ (necessarily $i = rs = \dim H/K$).  Then  $\Pi_{\sigma}$ contains
$\begin{pmatrix} n \\ r \end{pmatrix}$ members, indexed by permutations of the entries of $\lambda_{\sigma}$
or $\tilde{\lambda}_{\sigma}$, modulo (subsequent) permutations that preserve the first $r$ entries.  Write
$$\lambda_{\sigma} = (\lambda_1 > \lambda_2 > \dots > \lambda_n); ~~ \tilde{\lambda}_{\sigma} = (\tilde{\lambda}_1 > \tilde{\lambda}_2 > \dots > \tilde{\lambda}_n)$$
It's more convenient
to index $\Pi_{\sigma}$ by $(r,s)$-shuffles of the $\lambda_i$, namely $(r,s)$-tuples of the form
$\lambda = (a_1 > \dots > a_r; b_1 > \dots > b_s)$ which is a permutation of the entries of $\lambda_{\sigma}$.
If $\pi \in \Pi_{\sigma}$ corresponds to $\lambda$, then $\lambda$ is the {\it Harish-Chandra parameter} of $\pi$, and
we write $\pi = \pi_{\lambda}$.   We also define $\tilde{\lambda} = \lambda + \frac{n-1}{2}(1,1,\dots,1)$. 

Let $\Delta^{nc,+}$ be the set of 
positive non-compact roots -- roots of the form 
$$\alpha_{i,r+j}:  (c_1,c_2, \dots, c_n) \mapsto c_i - c_{r+j}, 1 \leq i \leq r, 1 \leq j \leq s.$$
These are by definition the roots of $T$ acting on $\fp^- \subset \mathfrak{h}$, the Lie algebra of matrices of the form
$ \begin{pmatrix} 0_r & X \\ 0 & 0_s \end{pmatrix}$, which is also isomorphic to the
antiholomorphic tangent space of the hermitian symmetric space $H/K$ at the fixed point of $K$. 
In other words, the choice of $\mathfrak{b}$, and therefore $\rho$,
is consistent with a specific choice of complex structure.
The {\it degree} $q(\lambda)$ of $\lambda$ is the number of pairs $(i,j)$ such that $a_i > b_j$, or equivalently
to the cardinality of 
$$\Delta^{nc,+}(\lambda) = \{\alpha \in \Delta^{nc,+} ~|~ \langle \alpha,\lambda \rangle > 0.$$
This is related to the usual notion of length by the formula
$q(\lambda) = rs - \ell(w)$, where $w$ is the permutation that takes $\lambda_{\sigma}$
to $\lambda$.
There are two extreme elements of $\Pi_{\sigma}$:  the {\it antiholomorphic representation}, with Harish-Chandra
parameter $\lambda_{\sigma}$ (and degree $rs$), and the {\it holomorphic representation}, corresponding to the longest shuffle, equivalently to the
Harish-Chandra parameter $(\lambda_{s+1} > \dots > \lambda_n;\lambda_1 > \dots > \lambda_s)$ (with degree $0$).  

The integer $q(\lambda)$ is the degree of coherent cohomology to which $\lambda$ contributes.  More precisely, to
any $\lambda$ we associate its {\it coherent parameter} $\Lambda = \lambda - \rho$, viewed as the highest weight
of an irreducible representation $W_{\Lambda}$ of $K$.   Let  $\fq = \fp^- \oplus Lie(K)$.  Then
\begin{equation}\label{dbar} \dim H^{q(\lambda)}(\fq,K;\pi_{\lambda}\otimes W_{\Lambda}^{\vee}) = 1 \end{equation}
and all other $H^i(\fq,K;\pi_{\lambda}\otimes W)$ vanish as $W$ runs over irreducible representations of $K$.

The space in \ref{dbar} is spanned by an element of $(\wedge^{q(\lambda)}(\fp^-)^*\otimes \pi_{\lambda}\otimes W_{\Lambda}^{\vee})^K$, or
equivalently by a homomorphism
\begin{equation}\label{hom} h_{\lambda} \in Hom_K(\wedge^{q(\lambda)}(\fp^-)\otimes W_{\Lambda}, \pi_{\lambda}). \end{equation}
The image of $h_{\lambda}$ is an irreducible $K$-type $\tau_{\lambda}$, the minimal (or Blattner) $K$-type, and its highest weight,
also denoted $\tau$, is the {\it Blattner parameter} of $\pi_{\lambda}$ (or of $\lambda$).   The formula for $\tau$ is given in terms of
the chamber $C(\lambda)$ for which $\lambda$ is positive
\begin{equation}\label{blatt} \tau = \Lambda + \sum_{\alpha \in \Delta^{nc,+}(\lambda)} \alpha. \end{equation}

For each $\alpha \in \Delta^{nc,+}$, choose a non-zero basis vector $X_{\alpha}$ in the corresponding
root space.

\begin{lem}\label{nonzero} The homomorphism $h_{\lambda}$ is non-zero on the vector
$$ v_{\lambda} = \wedge_{\alpha \in \Delta^{nc,+}(\lambda)} X_{\alpha} \otimes w_{\Lambda}$$
where $w_{\Lambda}$ is a basis vector of the highest weight subspace of $W_{\Lambda}$, and
$h(v_{\lambda})$ generates the highest weight space of $\tau$.
\end{lem}

\begin{proof}  This is presumably well-known but I include the one-line proof:  if we take a positive
root system for $K$ compatible with the chamber $C(\lambda)$, then $v_{\lambda}$, which is
a weight vector for the weight \ref{blatt}, is an extreme vector
in this chamber in the tensor product $\wedge^{q(\lambda)}(\fp^-)\otimes W_{\Lambda}$ and in particular
has multiplicity one.  Since the image of $h_{\tau}$ contains a vector of weight \ref{blatt}, it can't vanish
on $v_{\lambda}$.

\end{proof}

\subsection{Rationality}

Now we return to the notation of \S 3, with $G = GU(W)$, $G' = G(U(W') \times U(W_0))$.  The elements of the CM type
$\Sigma$ are denoted $v$, and for each $v$ the space $W_v = W \otimes_{F,v}\CC$ is a complex
hermitian space with signature $(r_v,s_v)$.  We assume that $W'_v$ has signature $(r_v - 1, s_v)$ for each $v$; 
although the discussion can be modified if this is not the case, we can always choose $W'$ with this property,
and it simplifies the discussion.  

A discrete series representation $\pi$ of $G(\RR)$ restricts to a sum of discrete series representations
of $H(\RR) = \prod_v U(r_v,s_v) := \prod_v H_v$.  In fact, there is no ambiguity unless $r_v = s_v$ for some $v$; each such
place contributes two summands, and we choose one of the two -- which we still call $\pi$ -- and replace $H_v$ at such a point by its
(real) identity component $H_v^+$.  We thus obtain a $\Sigma$-tuple of
Harish-Chandra parameters $(\lambda_v, v \in \Sigma)$.  Let $K_v = U(r_v,0)\times U(0,s_v) \subset U(r_v,s_v)$
be a maximal compact subgroup, as above, and let $K = \prod_v K_v$.  As above, the signature determines a holomorphic structure on the
the symmetric space $X(G)^+ = H(\RR)/K$ -- the product over $v$ of the spaces denoted $\fp^-$ above is the space
of anti-holomorphic tangent vectors at the fixed point of $K$ in the symmetric space -- and therefore we can assign
to each $v$ a reference parameter 
$$\lambda_{\sigma,v} = (\lambda_{1,v} > \dots > \lambda_{n,v})$$
which is the infinitesimal character of an irreducible finite-dimensional representation $(\sigma_v,V_v)$ of $H_v$.
This is a permutation of $\lambda_v$, which we write as above
$$\lambda_v = (a_{1,v}> \dots > a_{r_v,v}; b_{1,v} > \dots > b_{s_v,v}).$$
Let $\Lambda_v = \lambda_v - \rho_v$ be the $v$-component of the coherent parameter of $\pi$,
$\Lambda = (\Lambda_v, v \in \Sigma)$.
It defines an automorphic vector bundle $[W_{\Lambda}]$ on the Shimura variety $Sh(G,X(G))$, where $X(G)$ is the
(disconnected) hermitian symmetric space attached to $G(\RR)$ and $X(G)^+ \subset X(G)$ is a fixed 
$H(\RR)$-invariant component (also invariant under the identity component $G(\RR)^0 \subset G(\RR)$.  As in the
proof of \ref{mainthm}, we follow \cite{H90} and denote by $\bar{H}^*([W_{\Lambda}]$ the 
interior cohomology of $[W_{\Lambda}]$.  

Let $H' = U(W')\times U(W_0)$, $K' = \prod_v K'_v$, $K'_v = K_v \cap H'(\RR) = U(r_v - 1,0)\times U(1,0)\times U(0,s_v)$;
we drop the $0$'s in what follows.  
Write $\Lambda_v = (\alpha_{1,v} \geq \dots \geq \alpha_{r_v,v}; \beta_{1,v} \geq \dots \geq \beta_{s_v,v})$; this is a 
dominant integral weight of a finite-dimensional representation $W_{\Lambda_v}$ of $K_v$.  It follows from the
classical branching formula that the restriction of  $W_{\Lambda_v}$ to $K'_v$
contains the representation with highest weight
$$\Lambda'_v = (\alpha_{1,v} \geq \dots \geq \alpha_{r_v-1,v};\alpha_{r_v,v}; \beta_{1,v} \geq \dots \geq \beta_{s_v,v})$$
where the semicolons separate the weights for $U(r_v) \times U(1) \times U(s_v)$.  Let $\Lambda' = (\Lambda'_v)$ be
the corresponding highest weight for $K'$, and let $\lambda' = \Lambda' + \rho'$, where $\rho'$ is by analogy the half-sum of
the positive roots of $H'$ relative to $T$ for the chamber containing the chosen positive chamber for $(H,T)$.
With the above notation, we have
\begin{equation}\label{lambda'} \lambda'_v = (a_{1,v} - \frac{1}{2} > \dots > a_{r_v-1,v}-- \frac{1}{2}; \alpha_{r_v,v}; b_{1,v} + \frac{1}{2} > \dots > b_{s_v,v}+ \frac{1}{2}).\end{equation}

\begin{hyp}\label{regw}  The highest weight $a_{\sigma,v} =  \lambda_{\sigma,v} - \rho$ of $(\sigma_v,V_v)$ is a regular
character of $T$ for each $v$.   Equivalently, for each $v$, $\lambda_{i,v} - \lambda_{i+1,v} \geq 2$, $i = 1, \dots, n-1$.
\end{hyp}

\begin{lem}\label{HC}  Under Hypothesis \ref{regw}, $\lambda'$ is regular for the root system of $(H',T)$, and belongs to
$(\frac{n-2}{2} + \ZZ)^{n-1}$.  It is therefore 
the Harish-Chandra parameter for a unique discrete series representation $\pi_{\lambda'}$ of $H'(\RR)$.
\end{lem}

This is clear from the form of \ref{lambda'}:  under the regularity hypothesis, none of the
$a_{i,v} - \frac{1}{2}$ can coincide with any of the $b_{j,v} + \frac{1}{2}$.

We now consider the following hypothesis:

\begin{hyp}\label{knownknown}  For each $v$, the parameter $a_{r_v,v}$ equals $\lambda_{n,v}$, the smallest
of the $\lambda_{i,v}$.
\end{hyp}

Of the $\prod_v \begin{pmatrix} n \\ r_v \end{pmatrix}$ elements $\pi_{\lambda} \subset \Pi_{\sigma}$, the set of parameters 
satisfying \ref{knownknown} is 
$$\prod_v \frac{\begin{pmatrix} n-1 \\ s_v \end{pmatrix}}{\begin{pmatrix} n \\ s_v \end{pmatrix}} = \prod_v \frac{r_v}{n}$$
of the total.

Let $\Delta^{\prime,nc,+}$ be the set of positive non-compact roots for $(H',T)$, defined as in the previous
section.  We can view $\Delta^{\prime,nc,+}$ as the set of $\alpha_{i,r_v+j} \in \Delta^{nc,+}$ with $1 \leq i \leq r_v - 1; 1 \leq j \leq s_v$
for $v \in \Sigma$.  Hypothesis \ref{knownknown} is equivalent to the hypothesis that 
\begin{equation}\label{compat}  \Delta^{\prime,nc,+}(\lambda') = \Delta^{nc,+}(\lambda). \end{equation}

We define $X(G')^+ \subset X(G')$, $\fq'$, and the Shimura variety $Sh(G',X(G'))$, as in the previous paragraph; we also
define $\Lambda' = \lambda' - \rho'$, the representation $W_{\Lambda'}$ of $K'$, and the automorphic vector bundle $[W_{\Lambda'}]$ on 
$Sh(G',X(G'))$.
The inclusion of the Shimura data $(G',X(G')) \subset (G,X(G))$ defines an inclusion of Shimura varieties, and
by functoriality a restriction map of $\bar{\partial}$-cohomology
\begin{equation}\label{res}  ~~~ H^{q(\lambda)}(\fq,K;\pi_{\lambda}\otimes W_{\Lambda}^{\vee}) \ra H^{q(\lambda)}(\fq',K';\pi_{\lambda'}\otimes W_{\Lambda'}^{\vee}). \end{equation}
It follows immediately from Lemma \ref{nonzero} and \ref{compat} that
\begin{prop}\label{dbarcompat}   Assume Hypothesis \ref{regw}.   
\begin{itemize}
\item[(a)]  If $\lambda$ satisfies Hypothesis \ref{knownknown}, then the map (\ref{res}) is an isomorphism of $1$-dimensional spaces.
\item[(b)]  If $\lambda$ does not satisfy \ref{knownknown}, then (\ref{res}) is the zero map.
\end{itemize}
\end{prop}

For global automorphic representations $\pi$ and $\pi'$ whose
archimedean components satisfy Hypothesis \ref{knownknown}, the canonical period integral $I^{can}(\pi,\pi')$ can be identified with
cup products in coherent cohomology, by the formalism of \cite{H90,HL}.   The following theorem 
is an immediate consequence of Lemma 7.5.5 of \cite{H90} and the definitions:

\begin{thm}\label{GPknownknown}  Suppose $\lambda$ satisfies Hypotheses \ref{regw} and \ref{knownknown}.  Let $\pi$ and $\pi'$ be cuspidal automorphic
representations of $G$ and $G'$, respectively, with archimedean components $\pi_{\lambda}$ and $\pi_{\lambda'}^{\vee}$.
Then the Gross-Prasad period invariants $P_{\alpha}(\pi,\pi')$ are algebraic numbers.
\end{thm}

This can be refined to give vector-valued Gross-Prasad invariants, depending on complex embeddings of the coefficient fields,
whose entries belong to specific number fields.  Combining Theorem \ref{GPknownknown} with ergodic considerations, along
the lines of \ref{BSadelic}, we obtain examples of the rationality criteria provided by Theorem 7.6 and Corollary 7.7.1 of \cite{H90};
the point is that (a) of Proposition \ref{dbarcompat} implies hypothesis (b) of Theorem 7.6 of {\it loc. cit.}  Theorem \ref{GPknownknown}
is also of interest in proving period relations, and will be considered
at length in a future paper.  

\medskip

On the other hand, (b) of Proposition \ref{dbarcompat}, together with  Theorem \ref{mainthm}, gives explicit
examples of rationality
criteria that have nothing to do with coherent cohomology.   It is plausible that when $\pi$ does not
satisfy \ref{knownknown} then there is no $\pi'$ pairing non-trivially with $\pi$ for which
the $P_{\alpha}(\pi,\pi')$ are algebraic.  On the other hand, it is possible that $\pi^{\vee}$ satisfies
\ref{knownknown} even though $\pi$ does not.   Since a rationality criterion for $\pi^{\vee}$ easily gives
rise to one for $\pi$ (using cup products on $H$ rather than on $H'$), it is natural to wonder whether 
for all $\pi$ satisfying \ref{regw}, either $\pi$ or $\pi^{\vee}$ also satisfies \ref{knownknown}.  This seems
unlikely:  in the notation of \ref{knownknown}, suppose $a_{r_v,v} = \lambda_{n,v}$ for some but not all
$v$.  The above reasoning then suggests that, except possibly for some degenerate cases, neither $\pi$ nor 
$\pi^{\vee}$ satisfies \ref{knownknown}.


\begin{thebibliography}{99}






\bibitem [BS]{BS}  M. Burger and P. Sarnak, Ramanujan Duals II, 
{\it Invent. Math.}, {\bf 106} (1991), 1--11.
\medskip

\bibitem [C]{C} W. Casselman, Canonical extensions of Harish-Chandra modules to representations of $G$, 
{\it Can. J. Math.}, {\bf 41}, (1989) 385-438.


\bibitem [Cl]{Cl}  L. Clozel, Motifs et formes automorphes, in 
{\it Automorphic Forms, Shimura Varieties, and $L$-functions},
New York:  Academic Press (1990), Vol. 1, 77-159.  



\medskip
\bibitem[D]{D}  P. Deligne, Valeurs de fonctions $L$ et p\'eriodes d'int\'egrales,
{\it Proc. Symp. Pure Math.}, {\bf XXXIII}, part 2 (1979), 313-346. 

\medskip
\bibitem[FJ]{FJ}  M. Flensted-Jensen, Discrete series for semisimple symmetric spaces, 
{\it Ann. of Math.}, {\bf 111} (1980), 253-311. 

\medskip

\bibitem[GGP]{GGP}  W. T. Gan, B. Gross, and D. Prasad, Symplectic
local root numbers, central critical $L$-values, and restriction problems in the 
representation theory of classical groups, {\it Ast\'erisque}, {\bf 346} (2012).
\smallskip



\bibitem[H90]{H90}  M. Harris,    Automorphic forms of $\bar{\partial}$-cohomology type as
coherent cohomology classes,
{\it J. Diff. Geom.}, {\bf 32}, (1990) 1-63.
\medskip




\medskip
\bibitem[H12]{BB}  M. Harris,   Beilinson-Bernstein localization over $\QQ$ and periods of automorphic forms.
{\it Int. Math. Res. Not.} (2012)  doi: 10.1093/imrn/rns101
\medskip


\medskip
\bibitem [HL]{HL}  M. Harris and J.-S. Li,  A Lefschetz property for subvarieties of Shimura varieties. {\it J. Algebraic Geom.}, 
{\bf 7}  (1998) 77Ð122.



\medskip
\bibitem [HLS]{HLS}  M. Harris, J.-S. Li, B.-Y. Sun,  Theta correspondences for close unitary groups,
in J. Cogdell, J. Funke, M. Rapoport, T. Yang, eds., {\it Arithmetic Geometry and Automorphic Forms},
Beijing:  International Press (2011), 265--308.


\medskip
\bibitem [HS]{HS}  H. Hecht, W. Schmid,  On integrable representations of a semisimple Lie group,
{\it Math. Ann.}, {\bf 220} (1976)  147-149.












\medskip
\bibitem [L90]{L90}  J.-S. Li,  Theta liftings for unitary representations with non-zero cohomology,
{\it Duke Math. J.}, {\bf 61}  (1990)  913-937.








\medskip

\bibitem [SR]{SR}  S. Salamanca-Riba, On the unitary dual of real reductive groups and the $A_{\mathfrak{q}}(\lambda)$
modules:  the strongly regular case, {\it Duke Math. J.}, {\bf 96} (1998) 521--546.
\smallskip


\bibitem [SZ]{SZ}  B. Sun and C.-B. Zhu, Multiplicity one theorems:  the archimedean case.
\smallskip

\bibitem [Wa]{Wa} N. Wallach, On the constant term of a square integrable automorphic form,
in {\it Operator algebras and group representations}, Vol. II (Neptun, 1980), 227Ð237, {\it Monogr. Stud. Math.}, {\bf 18}, Boston:  Pitman  (1984).
\medskip


\end{thebibliography}
\end{document}